\documentclass[10pt]{amsart}
\usepackage{amsthm, amsfonts, amssymb, amsmath, amscd, graphicx, enumitem}
\usepackage[all]{xy}
\usepackage[margin=3.5cm]{geometry}

\newcommand{\C}{\mathbb{C}}

\newcommand{\G}{\Gamma}
\newcommand{\bd}{\partial}
\newcommand{\e}{\varepsilon}

\newcommand{\la}{\langle}
\newcommand{\ra}{\rangle}
\newcommand{\cleq}{\preccurlyeq}

\newtheorem*{theorem}{Theorem}
\newtheorem*{corollary}{Corollary}

\newtheorem{thm}{Theorem}

\newtheorem{lem}[thm]{Lemma}

\theoremstyle{definition}

\newtheorem*{ack}{Acknowledgments}
\newtheorem*{rem}{Remark}

\begin{document}
\title{On operator norms for hyperbolic groups} 
\author{Bogdan Nica}
\date{\today}
\address{Mathematisches Institut, Georg-August Universit\"at G\"ottingen}
\email{bogdan.nica@gmail.com}
\dedicatory{To the memory of Uffe Haagerup}

\begin{abstract} We estimate the operator norm of radial non-negative functions on hyperbolic groups. As a consequence, we show that several forms of Haagerup's inequality are optimal.
\end{abstract}
\maketitle

\section{Introduction}
Let $\G$ be a non-elementary hyperbolic group. As with any group, it is a difficult problem to compute, or even estimate the operator norm $\|\lambda(a)\|$, where $\lambda$ is the left-regular representation on $\ell^2\G$ and $a$ is an element in the group algebra $\C\G$. Haagerup's inequality gives an upper bound for the operator norm, and it is stated and used in either one of the following forms:
\begin{itemize}
\item[(H$_\circ$)] $\|\lambda(a)\|\cleq (n+1)\|a\|_2$ whenever $a\in \C\G$ is supported on the sphere of radius $n$; 
\item[(H$_\bullet$)] $\|\lambda(a)\|\cleq (n+1)^{3/2}\|a\|_2$ whenever $a\in \C\G$ is supported on the ball of radius $n$;
\item[(H$_s$)] $\|\lambda(a)\|\cleq \|a\|_{2,s}$ for $s> 3/2$, where $\|a\|_{2,s}=\big(\sum (|\gamma|+1)^{2s}|a(\gamma)|^2\big)^{1/2}$ is the $s$-weighted $\ell^2$-norm. 
\end{itemize}
The upper-bounding quantity requires the choice of a word-length $|\cdot|$, given by some finite and symmetric generating set for $\G$. The notation $\cleq$ means inequality up to positive multiplicative constants that only depend on $\G$ and the choice of generating set, and $\asymp$ is the corresponding equivalence.

Haagerup's original result \cite{Haa} concerned free groups, equipped with the standard word-length. The extension to hyperbolic groups is due to Jolissaint \cite{Jol} and de la Harpe \cite{dHa}. Note that the spherical inequality (H$_\circ$) implies the ball inequality (H$_\bullet$) and the weighted inequality (H$_s$), by obvious applications of the Cauchy - Schwartz inequality. The weighted inequality (H$_s$) is intimately related to certain Sobolev-type phenomena for the reduced group C$^*$-algebra of $\G$. The Haagerup inequality for hyperbolic groups is an instance of what we now call the \emph{property of Rapid Decay}. See \cite{Sap} for a very recent overview.
  
The starting point of this note is the following natural question: are the above Haagerup inequalities essentially sharp? For the spherical inequality (H$_\circ$), this is known to be the case for free groups. Cohen \cite{Coh} has computed the operator norm of the spherical element 
\begin{align*}
\sigma_n=\sum_{|\gamma|=n} \gamma\in \C\G
\end{align*} when $\G$ is free and endowed with the standard word-length. Cohen's computation immediately implies that $\|\lambda(\sigma_n)\|\asymp (n+1)\|\sigma_n\|_2$. This suggests that (H$_\circ$) ought to be essentially sharp for any hyperbolic $\G$, and it is fairly easy to devise a combinatorial argument - quasifying that of Cohen - showing that this is indeed the case. On the other hand, whether (H$_\bullet$) and (H$_s$) are essentially sharp is a less obvious matter. Relaxing the support condition on $a$ makes the operator norm of $a$ harder to estimate. But the fact that (H$_\bullet$) and (H$_s$) are reasonable by-products of (H$_\circ$), which we just claimed to be essentially sharp, gives a hint of what the answer might be.

The main result of this note is an estimate for the operator norm of radial elements in $\C\G$ with non-negative coefficients.

 \begin{theorem}
Let $a=\sum a_k\sigma_k\in \C\G$, where $a_k\geq 0$. Then $\|\lambda(a)\|\asymp \sum (k+1) \|a_k\sigma_k\|_2$.
\end{theorem}

\begin{corollary}
The following hold.
\begin{itemize}
\item[i)] For the spherical element $\sigma_n$, we have $\|\lambda(\sigma_n)\|\asymp (n+1)\|\sigma_n\|_2$. In particular, $\mathrm{(H_\circ)}$ is essentially sharp.
\item[ii)] For the ball element $\beta_n=\sum_{k\leq n} \sigma_k$, we have $\|\lambda(\beta_n)\|\asymp (n+1)\|\beta_n\|_2$.
\item[iii)] For the radial element $a=\sum_{k\leq n} a_k\sigma_k$, where $a_k\geq 0$ is such that $\|a_k\sigma_k\|_2=k+1$, we have $\|\lambda(a)\|\asymp (n+1)^{3/2}\|a\|_2$. In particular, $\mathrm{(H_\bullet)}$ is essentially sharp.
\item[iv)] For the radial element $a=\sum_{k\leq n} a_k\sigma_k$, where $a_k\geq 0$ is such that $\|a_k\sigma_k\|_2=(k+1)^{-2}$, we have $\|\lambda(a)\|\asymp \log (n+1)\: \|a\|_{2,3/2}$. In particular, $\mathrm{(H_s)}$ is essentially sharp.
\end{itemize}
\end{corollary}

The above statements directly follow from the main theorem, except for part ii) which needs in addition Lemma~\ref{third} below.

Parts iii) and iv) are more than sufficient to settle a problem raised in \cite[Section 4]{Nica}. The degree of rapid decay is $3/2$, not just for free groups as originally asked, but for hyperbolic groups in general.

\section{Proof of the main theorem}
That $\|\lambda(a)\|\cleq \sum (k+1) \|a_k\sigma_k\|_2$ follows from the spherical Haagerup inequality (H$_\circ$). We need to show that $\|\lambda(a)\|\succcurlyeq \sum (k+1) \|a_k\sigma_k\|_2$. A combinatorial argument, within $\G$, might be possible, but it appears to be difficult. We use instead an analytic approach, involving the boundary of $\G$. 

Consider the Cayley graph of $\G$ with respect to the given generating set, and let $\bd \G$ denote its boundary. The Gromov product on $\G$, given by the formula $(x,y)=\tfrac{1}{2}(|x|+|y|-|x^{-1}y|)$, is extended in a somewhat ad-hoc way to the boundary by setting
\begin{align*}
(x,\xi)=\inf_{x_i\to \xi}\: \liminf\: (x,x_i), \qquad (\xi,\xi')=\inf_{x_i\to \xi,  x_i'\to\xi'}\: \liminf\: (x_i,x_i').
\end{align*}
With this convention, the hyperbolic inequality $(x,y)\geq \min\{(x,z), (y,z)\}-\delta$ extends from $\G$ to $\G\cup\bd\G$.

We now recall some facts about the metric - measure structure of $\bd \G$. For small enough $\e>0$ there is a \emph{visual metric} $d_\e$ on the boundary $\bd \G$, satisfying $d_\e(\xi,\xi')\asymp \exp(-\e\: (\xi,\xi'))$. Different choices of the visual parameter $\e$ yield comparable Hausdorff measures. Let $\mu$ be a probability measure in the comparability class of visual Hausdorff measures. By \cite{Coo}, $\mu$ is Ahlfors-regular: the measure of any $r$-ball with respect to the visual metric $d_\e$ satisfies $\mu (r\textrm{-ball})\asymp r^{e(\G)/\e}$ for $0\leq r\leq \mathrm{diam}\:\bd\G$. Here, the \emph{critical exponent} $e(\G)$ with respect to the chosen word-length is the finite, positive number given by $e(\G)=\inf\{s>0: \sum_{\gamma} \exp(-s|\gamma|)<\infty\}$. For more details, see \cite{Coo} as well as the discussions in \cite[Sec.15]{KB}, \cite[Sec.5]{EN}.

The boundary representation of $\G$ with respect to $\mu$ is the unitary representation on $L^2(\bd \G,\mu)$ defined as follows:
 \begin{align*}
 \pi(\gamma)f=\big(d (\gamma_*\mu)/d\mu\big)^{1/2} \;\gamma.f
 \end{align*}

\begin{lem}\label{first}
The boundary representation $\pi$ is weakly contained in the regular representation $\lambda$.
\end{lem}
\begin{proof}
This follows \cite{AD, Nev} from the fact that the action of $\G$ on $\bd \G$ is amenable \cite{Ada}, \cite{Ger, Kai}.
\end{proof} 
Consequently, 
 \begin{align*}
 \|\lambda(a)\|\geq \|\pi(a)\|\geq \la\pi(a)\mathbf{1},\mathbf{1} \ra.
 \end{align*} 
where $\mathbf{1}$ is the constant function equal to $1$ on $\bd \G$.

The next step concerns the asymptotics of the spherical function $\gamma \mapsto \la\pi(\gamma)\mathbf{1},\mathbf{1} \ra$. In what follows, we put $q:= \exp(e(\Gamma))$.

\begin{lem}\label{second}
$\la\pi(\gamma)\mathbf{1},\mathbf{1} \ra\asymp (|\gamma|+1)\: q^{-|\gamma|/2}$.
\end{lem}

\begin{proof} By \cite{Coo}, $\mu$ is quasi-conformal: $d (\gamma_*\mu)/d\mu\: (\xi)\asymp \exp(e(\Gamma)\: (2(\gamma,\xi)-|\gamma|))$ for $\mu$-almost all $\xi\in \bd \G$. The ad-hoc definition of the extended Gromov product offers no guarantee that the map $\xi\mapsto (\gamma,\xi)$ is measurable. But the same procedure which yields the visual metric $d_\e$ on $\bd \G$ also yields a companion metric-like map $d_\e: \G\times \bd \G\to (0,\infty)$ satisfying $d_\e(\gamma,\xi)\asymp \exp(-\e (\gamma,\xi))$, as well as $|d_\e(\gamma,\xi)-d_\e(\gamma,\xi')|\leq d_\e(\xi,\xi')$. Thus, up to replacing it by a comparable continuous map, we may indeed assume that $\xi\mapsto q^{(\gamma,\xi)}$ is measurable. 

Now we have
\begin{align*}
\la\pi(\gamma)\mathbf{1},\mathbf{1} \ra=\int_{\bd \G} \big(d (\gamma_*\mu)/d\mu\big)^{1/2} \asymp q^{-|\gamma|/2} \int_{\bd \G} q^{(\gamma,\xi)} d\xi.
\end{align*}
To estimate the last integral, we apply the following general formula: if $(X,\nu)$ is a probability space and $f$ is a non-negative measurable map, then
\begin{align*}
\int_X \exp(f)=1+\int_0^\infty \nu\{x\in X: f(x)\geq t\} \exp(t) \: dt.
\end{align*}
In our case, $f$ is the map $\xi\mapsto (\log q)\: (\gamma,\xi)$. After a change of variable $t\mapsto (\log q)\: t$, we get
\begin{align*}
\int_{\bd \G} q^{(\gamma,\xi)}\: d\xi=1+(\log q) \int_0^\infty \mu(S(t))\: q^t \: dt
\end{align*}
where $S(t)=\{\xi\in \bd \G: (\gamma,\xi)\geq t\}$. We claim that the following hold:
\begin{itemize}
\item $S(t)$ is empty for $t>|\gamma|$; 
\item $\mu(S(t))\cleq q^{-t}$ for all $t\geq 0$;
\item there is a constant $C>0$, depending only on $\G$ and the choice of generating set, such that $\mu(S(t))\succcurlyeq q^{-t}$ for all $0\leq t\leq |\gamma|-C$.
\end{itemize}
These facts immediately yield that
\begin{align*}
\int_{\bd \G} q^{(\gamma,\xi)} d\xi\asymp |\gamma|+1.
\end{align*}
The first point is simply a rewriting of the fact that $(\gamma,\xi)\leq |\gamma|$ for all $\xi\in \bd\G$. For the second point, we note that $\mathrm{diam}\: S(t)\cleq \exp(-\e t)$, since $(\xi,\xi')\geq \min\{(\gamma,\xi), (\gamma,\xi')\}-\delta\geq t-\delta$ whenever $\xi,\xi'\in S(t)$. By Ahlfors-regularity, $\mu(S(t))\cleq  \exp(-\e t)^{e(\G)/\e}=q^{-t}$. Finally, let us settle the third point. There is a constant $C'>0$ such that each element of $\G$ has distance at most $C'$ to some geodesic ray based at the identity. Up to additive constants which can be absorbed into $C'$, the distance from $\gamma$ to a geodesic ray can be written as $|\gamma|-(\gamma,\omega)$, where $\omega\in \bd\G$ is the boundary point corresponding to the ray. Thus $(\gamma,\omega)\geq |\gamma|-C'$. Now let $t\leq |\gamma|-C'-\delta$. If $(\xi,\omega)\geq t+\delta$ then $(\gamma,\xi)\geq \min\{(\gamma,\omega), (\xi,\omega)\}-\delta\geq \min\{|\gamma|-C', t+\delta\}-\delta=t$. In other words, $S(t)$ contains a ball of radius $\asymp \exp (-\e t)$ around $\omega$. Ahlfors-regularity, once again, yields $\mu(S(t))\succcurlyeq  \exp(-\e t)^{e(\G)/\e}=q^{-t}$.
\end{proof}

The last ingredient is the asymptotical behaviour of spheres in $\G$.

\begin{lem}\label{third}
$\#\{\gamma \in \G: |\gamma|=k\}\asymp q^k$.
\end{lem}

\begin{proof}
This is, essentially, a result from \cite{Coo}.
\end{proof}

We now complete the proof of the theorem. In light of the last two lemmas, we have:
\begin{align*}
\la\pi(\sigma_k)\mathbf{1},\mathbf{1} \ra&=\sum_{|\gamma|=k} \la\pi(\gamma)\mathbf{1},\mathbf{1} \ra\asymp \sum_{|\gamma|=k} (k+1)\: q^{-k/2}\\
&\asymp (k+1)\: q^{k/2}\asymp (k+1)\:\|\sigma_k\|_2
\end{align*} 
Therefore
\begin{align*}
\|\lambda(a)\|\geq \la\pi(a)\mathbf{1},\mathbf{1} \ra = \sum a_k\: \la\pi(\sigma_k)\mathbf{1},\mathbf{1} \ra  \asymp \sum  (k+1)\:a_k\|\sigma_k\|_2
\end{align*} 
as desired.

 \begin{rem} Instead of using the boundary amenability, which yields Lemma 1, we can also argue as follows. By Lemma~\ref{second}, the spherical function $\gamma\mapsto \la\pi(\gamma)\mathbf{1},\mathbf{1} \ra$ is $p$-summable for every $p>2$. Let $\pi'$ be the restriction of $\pi$ to the closed subspace spanned by $\{\pi(\gamma)\mathbf{1}: \gamma\in \G\}$. Now \cite[Thm.1]{CHH} says that $\pi'$ is weakly contained in $\lambda$, and the rest of the argument goes through with $\pi'$ instead of $\pi$.
 \end{rem}

\begin{ack}
I thank Laurent Bartholdi for discussions, and for his interest in this paper.
\end{ack}

\end{document}